\documentclass[12pt]{article}
\usepackage{}
\usepackage{pifont}
\usepackage{enumerate}
\usepackage{amsfonts}
\usepackage{latexsym}
\usepackage{amsmath}
\usepackage{amssymb}
\usepackage{color}
\usepackage{makeidx}
\usepackage[pdftex,bookmarksnumbered,bookmarksopen,colorlinks,citecolor=blue,linkcolor=blue]{hyperref}
\usepackage{CJK}

\newcommand{\CA}{\mathcal{A}}
\newcommand{\CB}{\mathcal{B}}

\newcommand{\CF}{\mathcal{F}}

\newcommand{\CM}{\mathcal{M}}

\newcommand{\CP}{\mathcal{P}}
\newcommand{\CQ}{\mathcal{Q}}

\newcommand{\CT}{\mathcal{T}}
\newcommand{\CU}{\mathcal{U}}

\newcommand{\CZ}{\mathcal{Z}}

\newcommand{\mcal}{\mathcal}

\newcommand{\Aa}{\bigcap}
\newcommand{\Uu}{\bigcup}

\newcommand{\D}{\mathrm{d}}

\newcommand{\bde}{\begin{definition}}
\newcommand{\ede}{\end{definition}}
\newcommand{\beq}{\begin{eqnarray*}}
\newcommand{\eeq}{\end{eqnarray*}}
\newcommand{\bq}{\begin{eqnarray}}
\newcommand{\eq}{\end{eqnarray}}
\newcommand{\bbeq}{\begin{equation*}}
\newcommand{\eeeq}{\end{equation*}}
\newcommand{\bqq}{\begin{equation}}
\newcommand{\eqq}{\end{equation}}

\newtheorem{theorem}{Theorem}[section]

\newtheorem{proposition}{Proposition}[section]

\newenvironment{proof}[1][Proof]{\noindent \textbf{#1.} }{\ \ \  $\Box$}

\newtheorem{lemma}{Lemma}[section]
\newtheorem{definition}{Definition}[section]
\newtheorem{remark}{Remark}[section]

\title{{\LARGE\bf
        New Proofs for Several Properties of Capacities}}
\author{ \large Guangyan Jia\\
       Qilu Securities Institute for Financial Studies, Shandong University\\
       250199 Jinan, People's Republic of China\\
       E-mail: jiagy@sdu.edu.cn \\
       Na Zhang\footnote{corresponding author} \\
       College of Management and Economics, Tianjin University\\
China Center for Social Computing, Tianjin University\\
300072 Tianjin, People's Republic of China\\
       E-mail: znna1225@163.com}

\date{}

\begin{document}
\begin{CJK}{GBK}{song}

\maketitle
\par\noindent
-----------------------------------------------------------------------------------------------------------------------
\par\noindent
\vskip2mm{\bf Abstract}
In this note, we find a new way to prove several properties of 2-alternating capacities.

\vskip2mm\par {\bf Keywords:} Alternating capacity $\cdot$
       Probability measure $\cdot$ Monotone capacity $\cdot$ Minimal member
     \vskip2mm\par{\bf AMS 2010 subject classifications:\quad primary 28A12}

\par\noindent
-----------------------------------------------------------------------------------------------------------------------
\par\noindent
\section{Introduction}
Let $\Omega$ denote the basic set and $\CB$ the $\sigma$-algebra on $\Omega$.
A set function $c: \CB\rightarrow [0,1]$ is called a capacity if it satisfies:\\
(C1). $c(\Omega)=1, c(\emptyset)=0$;\\
(C2)(monotonicity). $c(A)\leq c(B)$ for any $A\subseteq B$, $A,B\in\CB$.\\
A capacity $c$ is called 2-alternating, if $c(A\cup B)+c(A\cap B)\leq c(A)+c(B)$.
It is called a probability measure if $c(A\cup B)+c(A\cap B)=c(A)+c(B)$. We usually denote a probability measure by $P$.

For any expectation $E$, we can define a capacity $c$ by $c(A)=E[I_A]$, $\forall A\in\CB$; on the other hand, for any capacity $c$, we can define expectation through Choquet integral, i.e., $E[X]=\int X\D c$. Choquet integral was first introduced by Choquet in 1953. The readers can refer to \cite{choquet1953} or \cite{denneberg1994} for more details. In \cite{denneberg1994}, Denneberg proved the following result.
\begin{lemma}[{\cite[Chapter 6]{denneberg1994}}]\label{denneberg1994chapter6}
If the integral with respect to a capacity $c$ is subadditive, i.e.,
\begin{eqnarray*}
\int (X+Y)\D c\leq \int X\D c+\int Y\D c,
\end{eqnarray*}
then $c$ is 2-alternating.
Conversely, let $c$ be a 2-alternating capacity, then for any $\CB$-measurable square integrable functions $X,Y$,
\begin{eqnarray*}
\int (X+Y)\D c \leq \int X\D c+\int Y\D c.
\end{eqnarray*}
\end{lemma}
In order to prove the above result, Denneberg proved the following result.
\begin{lemma}[{\cite[Lemma 6]{denneberg1994}}]\label{denneberg1994lemma6}
Suppose that $A_1,A_2,\ldots,A_n$ is a partition of $\Omega$, $\CB$ is a $\sigma$-algebra generated by $A_1,A_2,\ldots,A_n$ and $c:\CB\rightarrow[0,1]$ is a capacity. For any permutation $\pi$ of $(1,\ldots,n)$, we define
$$
S^\pi_i:=\bigcup\limits^i_{j=1}A_{\pi_j},\quad i=1,\ldots,n,\quad S^\pi_0:=\emptyset.
$$
We define a probability measure $P^\pi$ on $\CB$ by
$$
P^\pi(A_{\pi_i}):=\mu(S^\pi_i)-\mu(S^\pi_{i-1}),\quad i=1,\ldots,n.
$$
Suppose $X$ is a $\CB$-measurable real valued function $X$ defined on $\Omega$. If $\mu$ Is 2-alternating,
then
$$
\int X\D \mu\geq\int X\D P^\pi.
$$
If $X(A_{\pi_1})\geq X(A_{\pi_2})\geq\ldots\geq X(A_{\pi_n})$, the above equality holds.
\end{lemma}

Since Choquet integral is positive homogeneous, any Choquet expectation generated by a 2-alternating capacity is sublinear expectation. Jia \cite{jia2009minimal} defined a partial order "$\leq$" on the set of expectations as follows:
\begin{center}
for any two expectations $E_1$ and $E_2$, $E_1\leq E_2$ if for any $\CB$-measurable square integrable random variable $X$, $E_1[X]\leq E_2[X]$,
\end{center}
and proved the following results.
\begin{lemma}[{\cite[Theorem 2.7 and Theorem 3.1]{jia2009minimal}}]
$E$ is a minimal member of the set of all the sublinear expectations if and only if $E$ is a linear expectation. Suppose $E_1$ is a subadditive expectation, $E_2$ is a superadditive expectation and $E_1\geq E_2$, then there exists a linear expectation $E_0$ such that $E_1\geq E_0\geq E_2$.
\end{lemma}
Therefore we may wonder if the minimal members of all the 2-alternating capacities are exactly the probability measures? In fact, the answer is positive, since the following result holds:
\begin{center}
Suppose $\CQ$ is a set, $\prec$ is a semiorder defined on $\CQ$ and $\CU$ denotes the set of all the minimal members in $\CQ$. Thus for any set $\CZ$ satisfying $\CU\subset\CZ\subset\CQ$, the set of all the minimal members of $\CZ$ is still $\CU$.
\end{center}

In this note, we'll give another proof of the above results by means of capacity only.
The method of constructing a probability measure step by step from a 2-alternating capacity is also given.

\section{Main results}
First we list the following definitions which will be used below.
A capacity defined on $(\Omega,\CB)$ is said to be:
\begin{itemize}
\item 2-$monotone$ if $c(A\cup B)+c(A\cap B)\geq c(A)+c(B)$;
\item $n$-$alternating$ if $c(\Aa\limits^n_{i=1}A_i)\leq \sum\limits_{\emptyset\neq I\subset\{1,...,n\}}(-1)^{|I|+1}v(\Uu\limits_{i\in I}A_i), \forall A_1,...,A_n\in \CB$;
\item $n$-$monotone$ if $c(\Uu\limits^n_{i=1}A_n)\geq \sum\limits_{\emptyset\neq I\subset\{1,...,n\}}(-1)^{|I|+1}v(\Aa\limits_{i\in I}A_i), \forall A_1,...,A_n\in \CB$;
\item $\infty$-$alternating$ if $c$ is n-alternating, for all $n$;
\item $\infty$-$monotone$ if $c$ is n-monotone, for all $n$.
\end{itemize}
Furthermore, we have the following notations.
\begin{itemize}
\item $\CA_n$ denotes the set of $n$-alternating capacities, for any $n\geq 2$;
\item $\CM_n$ denotes the set of $n$-monotone capacities, for any $n\geq 2$;
\item $\CP$ denotes the set of probability measures;
\item $\CA_{\infty}$ denotes the set of $\infty$-alternating capacities;
\item $\CM_{\infty}$ the set of $\infty$-monotone capacities.
\end{itemize}
It is known that $\CP\subseteq\CA_{\infty}\subseteq\CA_{n+1}\subseteq \CA_n$, $\CP\subseteq\CM_{\infty}\subseteq\CM_{n+1}\subseteq \CM_n$ and $\CA_n\cap\CM_m=\CP$ for any $n\geq 2$ and $m\geq 2$.

Now let us define the partial order "$\leq$": for any two capacities $c_1$ and $c_2$, $c_1\leq c_2$ means that $c_1(A)\leq c_2(A)$, for all $A\in \CB$. If $c_1\leq c_2$, we can also denote by $c_2\geq c_1$. If $c_1\leq c_2$ and $c_1\geq c_2$, we have $c_1=c_2$.

The following lemma holds.
\begin{lemma}
Let $\mcal{T}\subset \CA_2$ be a nonempty set and totally ordered (for each pair $c_1,c_2\in \mcal{T}$, one has either $c_1\leq c_2$ or $c_2\leq c_1$). Then the set function
$$
\nu(A)\triangleq\inf\limits_{c\in\mcal{T}}c(A),\qquad A\in \CB,
$$
is a 2-alternating capacity, that is $\nu\in \CA_2$.
\end{lemma}
\begin{proof}
It is obvious that $\nu(\Omega)=1,\nu(\emptyset)=0$ and $\nu$ is monotone. We now prove that it is 2-alternating.
\begin{eqnarray*}
\nu(A\cap B)=\inf\limits_{c\in\CT}c(A\cap B)\\
\leq\inf\limits_{c\in\CT}\{c(A)+c(B)-c(A\cup B)\}\\
\leq \inf\limits_{c\in\CT}\{c(A)+c(B)\}-\nu(A\cup B).
\end{eqnarray*}
Since $\CT$ is totally ordered, for $c_1,c_2\in \CT$, we suppose, without lost of generality, that $c_1\leq c_2$, so $c_1(A)+c_2(B)\geq c_1(A)+c_1(B)$.
Therefore,
\begin{eqnarray*}
\nu(A\cap B)\leq \inf\limits_{c_1,c_2\in\CT}\{c_1(A)+c_2(B)\}-\nu(A\cup B)\\
=\inf\limits_{c\in\CT}\{c(A)\}+\inf\limits_{c\in\CT}\{c(B)\}-\nu(A\cup B)\\
=\nu(A)+\nu(B)-\nu(A\cup B).
\end{eqnarray*}
Thus the result follows.
\end{proof}

\begin{theorem}\label{main}
Any $P\in\CP$ is a minimal member of $\CA_2$. Conversely, if $c$ is a minimal member of $\CA_2$, then $c\in \CP$.
\end{theorem}
\begin{proof}
Suppose $c\in\CA_2, c\leq P$. Then we have
$$
\forall A\in \CB, 1-c(A^c)\leq c(A)\leq P(A)=1-P(A^c).
$$
Since $c(A^c)\leq P(A^c)$, we have $c(A)=P(A)$, which means that there is no 2-alternating capacity $c$ satisfying $c\leq P$, i.e., $P$ is a minimal member of $\CA_2$.

If $c$ is a minimal member of $\CA_2$, for a fixed $A\in\CB$, we define
$$
c^A(B):=c(A\cup B)+c(A\cap B)-c(A), \forall B\in\CB.
$$
Obviously, $c^A\leq c$, $c^A(\Omega)=c(\Omega)+c(A)-c(A)=1$, $c^A(\emptyset)=c(A)+0-c(A)=0$. The monotonicity of $c^A$ can be easily deduced by the monotonicity of $c$.
For any $B\in \CB,F\in \CB$,
\begin{eqnarray*}
c^A(B\cup F)+c^A(B\cap F)
&=&c(A\cup(B\cup F))+c(A\cap(B\cup F))-c(A)\\
&&+c(A\cup(B\cap F))+c(A\cap(B\cap F))-c(A)\\
&=&c((A\cup B)\cup(A\cup F))+c((A\cap B)\cup(A\cap F))\\
&&+c((A\cup B)\cap(A\cup F))+c((A\cap B)\cap(A\cap F))-2c(A)\\
&\leq & c(A\cup B)+c(A\cup F)+c(A\cap B)+c(A\cap F)-2c(A)\\
&=& c^A(B)+c^A(F),
\end{eqnarray*}
i.e., $c^A\in \CA_2$. Note that $c$ is the minimal member of $\CA_2$, thus $c^A=c$, which means that, for any $B\in \CB$, we have
$c(A\cup B)+c(A\cap B)=c(A)+c(B)$. Since $A$ can be any set in $\CB$, $c$ is a probability measure.
\end{proof}

\begin{remark}
\begin{enumerate}
\item By similar proof we can deduce that any minimal member of $\CA_n$ ($n\geq 2$) (resp. $\CA_\infty$) can only be probability measure and any probability measure is its minimal member.
\item The maximal member of $\CM_n$ ($n\geq 2$) (resp. $\CM_\infty$) can only be probability measure and any probability measure is its maximal member.
\end{enumerate}
\end{remark}

\begin{definition}
For a capacity $c$, we define the invariant subfield $\CB^c$ of $c$ as follows:
$$\CB^c\triangleq\{A\in\CB: \forall B\in\CB, c(A\cup B)+c(A\cap B)=c(A)+c(B)\}$$
\end{definition}

It is obvious that $\CB^c$ is nonempty, since $\Omega\in\CB^c$ and $\emptyset\in\CB^c$. A capacity $c$ is a probability measure if and only if $\CB^c=\CB$.
Note that if $c$ is a 2-alternating capacity, then for all $A\in\CB$, such that $c(A)=0$, we have $A\in\CB^c$. If $c$ is a 2-monotone capacity, then for all $A\in\CB$, such that $c(A)=1$, we have $A\in \CB^c$.

$\forall c\in\CA_2$, it has been proved that $c^F\in \CA_2$. Thus we can define the following mapping.

\begin{definition}
For all $\CF\in\CB$, we define mapping $\Pi^F:\CA^2\rightarrow\CA^2$ as follows:
$$
\Pi^Fc=c^F.
$$
\end{definition}

\begin{proposition}\label{cap:prop:mapping}
The following properties about invariant subfield and the above mapping hold.
\begin{enumerate}[(i).]
\item $\forall c\in\mathcal{A}_2$, $\Pi^F(c)\leq c$;
\item $\forall A\in \CB$, if $A\subset F$ or $F\subset A$, one has $c^F(A)=c(A)$;
\item $\forall A\in \CB^c$, $c^F(A)=c(A)$;
\item  $F\in\CB^{c^F}$;
\item $\CB^{c}\subset \CB^{c^F}$;
\item If $F\in\CB^c$, $c^F=c$.
\end{enumerate}
\end{proposition}
\begin{proof}
(i) and (vi) are obvious.

(ii).
Without lost of generality, suppose that $A\subset F$, thus
$c^F(A)=c(A\cup F)+c(A\cap F)-c(F)=c(F)+c(A)-c(F)=c(A)$.

(iii).
For all $A\in\CB^c$,
$c^F(A)=c(F\cup A)+c(F\cap A)-c(F)=c(A).$

(iv).
According to (ii), for all $A\in\mathcal{B}$,
\begin{eqnarray*}
&c^F(F\cup A)+c^F(F\cap A)-c^F(F)-c^F(A)\\
&=c(F\cup A)+c(A\cap F)-c(F)-(c(A\cup F)+c(A\cap F)-c(F))=0.
\end{eqnarray*}

(v).
Suppose $A\in \CB^c$, $B\in\CB$,
\begin{eqnarray*}
&&c^F(A\cup B)+c^F(A\cap B)-c^F(A)-c^F(B)\\
&=&c(F\cup (A\cup B))+c(F\cap (A\cup B))-c(F)+c(F\cup (A\cap B))\\
&&+c(F\cap (A\cap B))-c(F)-c(A)-c(F\cup B)-c(F\cap B)+c(F)\\
&=&[c(A)+c(F\cup B)-c(A\cap (F\cup B))]+c(F\cap (A\cup B))-c(F)\\
&&+c(F\cup (A\cap B))+[c(A)+c(F\cap B)-c(A\cup (F\cap B))]\\
&&-c(F)-c(A)-c(F\cup B)-c(F\cap B)+c(F)
\end{eqnarray*}
\begin{eqnarray*}
&=&c(A)-c(A\cap (F\cup B))+c(F\cap (A\cup B))-c(F)+c(F\cup (A\cap B))\\
&&-c(A\cup (F\cap B))\\
&=&[c(F\cup (A\cap B))-c(A\cap (F\cup B))]\\
&&+[c(F\cap (A\cup B))-c(A\cup (F\cap B))]+c(A)-c(F)\\
&=&[c(F\cup (A\cap B))-c((F\cup (A\cap B))\cap A)]\\
&&+[c(F\cap (A\cup B))-c((F\cap (A\cup B))\cup A)]+c(A)-c(F)\\
&=&[c((F\cup (A\cap B))\cup A)-c(A)]+[c((F\cap (A\cup B))\cap A)-c(A)]\\
&&+c(A)-c(F)\\
&=&[c(F\cup A)-c(A)]+[c(F\cap A)-c(A)]+c(A)-c(F)=0.
\end{eqnarray*}
\end{proof}

With the help of this mapping, we can prove Lemma \ref{denneberg1994lemma6}, i.e., the following theorem, by way of capacity.
\begin{theorem}\label{exist}
Consider $(\Omega,\CB)$. Suppose that $\CB$ is finite, $c$ is a 2-alternating capacity defined on $\CB$. Take $F_1,...,F_n\in \CB$ such that $F_1\subset F_2\subset...\subset F_n$. Thus there exists a probability measure $P$, such that $P(F_i)=c(F_i)$, for all $i=1,...,n$ and $P\leq c$.
\end{theorem}

\begin{proof}
First, we design a cyclic program as follows.

Set $\mu=c$.

{\bf Step I}: Check $F_i$, $i=1,...,n$. If all the sets $F_i$ belong to the invariant subfield of $\mu$, go straight to Step III. Otherwise, suppose that $F_i$ does not belong to the subfield of $\mu$. By Proposition \ref{cap:prop:mapping}, the following result holds:
$$\mu^{F_i}(F_j)=\mu(F_j),\forall j=1,...,n,$$
i.e., $\mu^{F_i}$ and $\mu$ are equal on $F_j$, $j=1,...,n,$;
$$F_i\in\CB^{\mu^{F_i}},\quad
\CB^\mu\subset\CB^{\mu^{F_i}},$$
i.e., from $\mu$ to $\mu^{F_i}$, the invariant subfield is enlarged and $F_i$ is also included.

{\bf Step II}: Update $\mu$ by $c^{F_i}$. The invariant subfield of $\mu$ is enlarged by Step I. Repeating the procedures in Step I.

{\bf Step III}: We get the final $\mu$, which satisfies $\mu\in\CA_2$, $\mu\leq c$, and for all $i=1,...,n$, $\mu(F_i)\equiv c(F_i)$, $F_i\in\CB^\mu$, $\CB\subset\CB^\mu$.

Next, we consider $\mu$, and design another cyclic program.

{\bf Step 1}: Check $\CB^{\mu}$ and $\CB$. If they are the same, go straight to Step 3. Otherwise, suppose $A\in\CB/\CB^{c_F}$. Consider the transformation of $\mu$ induced by $A$. By Proposition \ref{cap:prop:mapping}, we have
$\mu^A(F_i)\equiv\mu(F_i)$, $i=1,...,n$.

{\bf Step 2}: Update $\mu$ by $\mu^A$. The invariant subfield of $\mu$ is enlarged. Repeat the procedures in Step 1.

{\bf Step 3}: $\mu$ satisfies the following conditions: for all $i=1,...,n$, $\mu(F_i)=c(F_i)$, $\mu\leq c$. Furthermore, $\CB^{\mu}=\CB$, thus $\mu$ is just the probability measure satisfying the conditions needed.
The proof is complete.
\end{proof}

\begin{theorem}\label{cap:thm:sandwich}
Consider space $(\Omega,\CB)$. Suppose that $\CB$ is finite, $\mu$ is a 2-alternating capacity defined on $\CB$, $\nu$ is a 2-monotone capacity defined on $\CB$. If $\mu\geq \nu$, there exists a probability measure $P$ such that $\mu\geq P\geq \nu$.
\end{theorem}
\begin{proof}
Since $\CB$ is finite, we can take $A\in\CB/\CB^{\mu}$ such that
 $$\mu(A)-\nu(A)=\min\limits_{B\in\CB/\CB^{\mu}}\left\{\mu(B)-\nu(B)\right\}.$$
Make transformation $\Pi^A$ on $\mu$, thus
$$\mu^A(B)=\mu(A\cup B)+\mu(A\cap B)-\mu(A)\geq \nu(A\cup B)+\nu(A\cap B)-\nu(A)\geq \nu(B),$$
i.e., $\mu\geq\mu^A\geq \nu$. By Proposition \ref{cap:prop:mapping},
$$
\CB^{\mu}\subset\CB^{\mu^A},
$$
and $\CB^{\mu}\neq\CB^{\mu^A}$.

For $\mu^A$, repeat the above procedure, until we get a capacity $P$, such that $\CB^{P}=\CB$. $P$ satisfying that $\mu\geq P\geq \nu$. The proof is complete.
\end{proof}

\begin{remark}
According the above theorem, we may get different probability measures if we make transformation by different sets or in a different order.
\end{remark}

\end{CJK}

\begin{thebibliography}{999}
\bibitem{choquet1953}Choquet, G., Theory of capacities, Annales de l'Institut Fourier, 5 (1953-1954), 131-295.
\bibitem{denneberg1994} Dieter Denneberg, Non-additive measure and integral, Kluwer Academic Publishers, Boston,1994, 184 pp.
\bibitem{jia2009minimal} Guangyan Jia, The minimal sublinear expectations and their related properties, Sicence in China Ser. A: Mathematics. 39(2009), 79-87.
\end{thebibliography}
\end{document}